\documentclass[11pt,letterpaper]{amsart}
\usepackage[utf8]{inputenc}
\usepackage{verbatim}
\usepackage{amsmath,amssymb,amsfonts,amsthm}
\usepackage{bigints}
\usepackage{hyperref}
\usepackage{array}
\numberwithin{equation}{section}

\usepackage{verbatim}
\usepackage{amsmath,amssymb,amsfonts,amsthm}
\usepackage{bigints}
\usepackage{hyperref}
\usepackage{array}

\usepackage{amsthm}

\usepackage{mathtools}

\mathtoolsset{showonlyrefs}

\newtheorem{definition}{Definition}[section]
\newtheorem{proposition}[definition]{Proposition}
\newtheorem{theorem}[definition]{Theorem}
\newtheorem{conjecture}[definition]{Conjecture}

\newtheorem{lemma}[definition]{Lemma}

\newcommand\R{\mathbb{R}}
\newcommand\N{\mathbb{N}}
\newcommand\Z{\mathbb{Z}}

\newcommand{\norm}[1]{ \left\lVert#1\right\rVert}

\newcommand{\spec}{\operatorname{spec}}

\title{Lacunary}
\date{}

\begin{document}

\author{Miquel Saucedo}
\address{M.  Saucedo,  Centre de Recerca Matemàtica\\
Campus de Bellaterra, Edifici C
08193 Bellaterra (Barcelona), Spain}
\email{miquelsaucedo98@gmail.com }

\author{Sergey Tikhonov}
\address{S. Tikhonov, 
ICREA, Pg. Lluís Companys 23, 08010 Barcelona, Spain;
Centre de Recerca Matem\`{a}tica\\
Campus de Bellaterra, Edifici C
08193 Bellaterra (Barcelona), Spain;
 and Universitat Autònoma de Barcelona.}
\email{stikhonov@crm.cat}

\title[A Logvinenko-Sereda theorem for lacunary spectra ]{A Logvinenko-Sereda theorem for 
 lacunary spectra
}

\keywords{Fourier transform,  uncertainty principles}

\subjclass{Primary: 42B10,42A55}

\thanks{This work is supported by
 PID2023-150984NB-I00 funded by MICIU/AEI/10.13039/501100011033/ FEDER, EU, 
the CERCA Programme of the Generalitat de Catalunya and the Severo Ochoa, and Mar\'ia de Maeztu
Program for Centers and Units of Excellence in R\&d (CEX2020-001084-M).
M. Saucedo is supported by  the Spanish Ministry of Universities through the FPU contract FPU21/04230.
 S. Tikhonov is supported
by 2021 SGR 00087. }

\begin{abstract}
For a function $F$ represented as
$F(x)=\sum_{n=0}^\infty{f_n (x) e^{2 \pi i \lambda_n x}},$ 
where
 each 
 $f_n$
 satisfies
$\spec(f_n) \subset [0, 1]$ and 
$(\lambda_n)_{n\geq 0}\subset \R_+$ is a lacunary sequence, we obtain
$$ \|F\|_{L^2(\mathbb{R})}\lesssim
\|F\chi_{E}\|_{L^2(\mathbb{R})}
$$
provided that $E$ is
a thick subset of $\mathbb{R}$.
This extends the Logvinenko-Sereda theorem and answers 
a
question posed by 
Kovrizhkin for functions  with positive frequencies.
\end{abstract}
\maketitle
\section{Introduction}
The uncertainty principle in Fourier analysis forbids a function and its Fourier
transform from being simultaneously concentrated on small sets. There are many results of this type, see for instance the book \cite{HavinJoricke}.

In this work, we consider quantitative statements of the uncertainty principle  of the type
\begin{equation}
\int_{\mathbb{R}} |f|^2
\le C\left(
\int_E |f|^2
+
\int_F |\widehat f|^2
\right) \mbox{ for } f\in L^2(\R);
\end{equation}
or, equivalently,
\begin{equation}\label{ls}
\int_{\mathbb{R}} |f|^2
\le C \int_E |f|^2 \mbox{ for } \widehat{f}\in L^2(F^c),
\end{equation}
where $E$ and $F$ are the complements of ``small'' sets. Observe that these inequalities imply, in particular, that if $f\in L^2(\R)$ is supported on $E^{c}$ and $\widehat f$ is supported
on $F^{c}$, then $f \equiv 0$.

The well-known 
Logvinenko--Sereda result characterizes the sets $E$ for which inequality \eqref{ls} holds for $F^c=[0,1]$; see, e.g., \cite[Chapter 3, \S 4]{HavinJoricke}.
\begin{theorem}\label{th:logvinenko}
 Let $F^c=[0,1]$. Then inequality \eqref{ls} holds for a measurable set $E$ if and only if $E$ is ($\Delta,\gamma$)-\textit{thick} (or, simply, thick), that is, if and only if
there exist $\Delta,\gamma>0$ such that
for every interval $I$ of length $\Delta$,
$$|E \cap I| \geq \gamma |I|.$$ Moreover, one can take $C:=C(\gamma,\Delta)$ in inequality \eqref{ls}.
\end{theorem}

The Logvinenko-Sereda theorem has been generalized in several ways:
\begin{enumerate}
    \item[(i)] In \cite{kovrinolac}, Kovrizhkin showed that Theorem \ref{th:logvinenko} holds true for $F^c=\cup_{n=1}^N I_n$, where $I_j$ are intervals of length $1$, and with $C:=C(\gamma, \Delta, N)$.
    \item[(ii)] In \cite{Kovrizhkin}, he extended Theorem \ref{th:logvinenko} to the case where $F^c$ is a lacunary set of intervals; under the additional conditions that $\gamma$ is sufficiently close to $1$ or that the intervals are sufficiently small. Moreover, he conjectured that these additional conditions are not needed. More precisely, he formulated the following conjecture.
\begin{conjecture}\label{conj:Kovrizhkin}
Let $\Lambda=\{\lambda_k\}_{k=1}^{\infty}$ be Zygmund lacunary. Then Theorem \ref{th:logvinenko} holds true for $$F^c=\bigcup_{k=1}^{\infty}
\left[\lambda_k,\;\lambda_k+1\right].$$
\end{conjecture}
We recall that a sequence $\Lambda\subset \mathbb{R}$ is said to be Zygmund lacunary  if 
$$\sup_{k\neq l}\# \left\{ (k',l'), |\lambda_{k} -\lambda_l -(\lambda_{k'}-\lambda_{l'})|\leq 1\right\}=N<\infty.$$
\item[(iii)] In \cite{jaye1}, the authors generalized Theorem \ref{th:logvinenko} to a wide class of sparse sets $F^c$ under the additional condition that $E$ is a $\sigma$-neighborhood of a thick set.
\item[(iv)]
Finally, it is worth mentioning that there are also works addressing extensions of Theorem~\ref{th:logvinenko} with a geometric component, such as \cite{jaye2} and \cite{jaye2025highfrequencyuncertaintyprinciplefourierbessel}.

%Finally, it is worth mentioning that there are also works addressing  extensions of Theorem \ref{th:logvinenko} with a geometric component, such as \cite{jaye2} and \cite{jaye2025highfrequencyuncertaintyprinciplefourierbessel}.

\end{enumerate} 

In this paper we confirm Conjecture \ref{conj:Kovrizhkin} a class of lacunary sequences.
\subsection{Main result}
Before formulating our main theorem we need the following definition.
\begin{definition}[Strong Zygmund lacunary]
We say that a sequence $\Lambda=
(\lambda_k)_{k=1}^\infty \subset \R$
 is (M,N)-strong Zygmund lacunary (or, simply, strong Zygmund lacunary) if there exists a function $M: \N\to \N$ such that for any $L\in \N $ the sequence $$(\lambda_k/L)_{|k| \geq M(L)}$$ is Zygmund lacunary with constant $N$ independent of $L$, that is, if 
$$\sup_{\substack{k\neq l\\k,l\geq M(L)}}\# \left\{ (k',l'), |\lambda_{k} -\lambda_l -(\lambda_{k'}-\lambda_{l'})|\leq L\right\}=N<\infty.$$
\end{definition}
Note that all Hadamard lacunary sequences, that is, those satisfying the condition $\lambda_{k+1}/\lambda_k\geq q>1$, are strong Zygmund lacunary. On the other hand, there are strong Zygmund lacunary sequences which are not Hadamard lacunary; see Proposition \ref{pr:nothadamard}.

 After giving this definition, we are in a position to formulate our main result.
 \begin{theorem}
 \label{theorem:mainth}
     Conjecture \ref{conj:Kovrizhkin} is true for any strong Zygmund lacunary sequence $\Lambda$ under the additional condition that $\Lambda \subset \mathbb{R}_+$.
 \end{theorem}

\subsection{Structure of the paper}
After the proof of Theorem~\ref{theorem:mainth} in Section~2, 
we show in Section \ref{uniq sets} that any open set is a set of uniqueness for functions whose spectrum lies in 
$
 \bigcup_{n \in \mathbb{N}} 
[-\lambda_n - 1, -\lambda_n] \cup [\lambda_n, \lambda_n + 1],
$
under suitable lacunarity conditions on the sequence $(\lambda_n)$.

Finally, in Section~\ref{section:4}, 
we prove that the class of strong Zygmund lacunary sequences is strictly intermediate between the classes of Hadamard and Zygmund lacunary sequences.

\section{Proof of Theorem \ref{theorem:mainth}}
Our main auxiliary result is the following
\begin{lemma}\label{lemmamain}
Let $\Lambda$ be a strong Zygmund lacunary sequence. Let $L\geq 1, 0 < \gamma < 1$. Let $E\subset \mathbb{R}$ be a set of positive measure and $I$ an interval with $|I | = {L^{-1}}$ and
$$|E \cap I| \geq {{\gamma}\over {L}}.$$
Then there are two constants $C_1(\gamma)$ and $C_2$ such that, for any sequence of smooth functions $f_n$, \begin{align*}
    \int_{I \cap E} {\bigg| \sum_{|n|\geq M(L)} f_n(x) e^{2 \pi i \lambda_n x} \bigg|^2 dx} &\geq C_1(\gamma) \int_{I}{\sum_{|n|\geq M(L)}{|f_n(x)|^2}} \\&- {{C_2}\over{L^{\frac12}}} \int_{I}{\sum_{|n|\geq M(L)}{\Big(|f_n(x)|^2 + |f_n'(x)|^2\Big)dx}}.
    \end{align*}
\end{lemma}
We will make use of the following result obtained by Nazarov  \cite{nazarov}.
    
\begin{theorem}\label{th:nazarov}

    Let 
$E\subset \mathbb{T}$ be of positive measure, and let 
$\Lambda$ be a Zygmund lacunary sequence.   
    Then
$$C(|E|,N)\sum_{n \in \Z} {|a_n|^2} \leq \int_{E}{\Big|\sum_{n \in \Z}{a_n e^{2 \pi i \lambda_n x}}\Big|^2 dx }.
$$

\end{theorem}

\begin{proof}[Proof of Lemma \ref{lemmamain}]
Fix an interval $I=[0, L^{-1}]$ and  $y \in I$ to be chosen later. We have
\begin{align*} &\int_{I \cap E}{ {\bigg| \sum_{|n|\geq M(L)} f_n(x) e^{2 \pi i \lambda_n x} \bigg|^2 dx} } = 
\int_{I \cap E}\Big( \sum_{|n|,|m| \geq M(L)}{f_n(x) \overline{f_m}(x) e^{2 \pi i (\lambda_n - \lambda_m )x}} \Big)dx\\
&=\int_{I \cap E} \sum_{|n|,|m| \geq M(L)} {f_n(y) \overline{f_m}(y) e^{2 \pi i (\lambda_n - \lambda_m )x}}dx
\\
&+\int_{I \cap E} \sum_{|n|,|m| = M(L)}\Big(\int_{y}^{x}{(f_n \overline{f_m})'(t)}
dt\Big) e^{2 \pi i (\lambda_n - \lambda_n)x}dx
 =: K_1 + K_2.
   \end{align*}
   
First,
\begin{align*} K_1 &= \int_{I \cap E} \bigg| \sum_{|n|\geq M(L)} f_n(y) e^{2 \pi i \lambda_n x} \bigg|^2 dx =
\int_{0}^{{{1}/{L}}} \bigg| \sum_{|n|\geq M(L)} f_n(y) e^{2 \pi i \lambda_n x} \bigg|^2 {\chi_{E}(x)}dx 
\\
&= \int_{0}^{1} \bigg|{\sum_{|n|\geq M(L)}}f_n(y) e^{2 \pi i \frac{\lambda_n}{L} x} \bigg|^2 \chi _{E} {({\frac{x}{L}})}{dx \over L}.
 \end{align*}
Applying Nazarov's result (Theorem \ref{th:nazarov}) to the Zygmund lacunary sequence $(\lambda_n/L)_{|n| \geq M(L)}$ and noting that 
$$ 
\int_{0}^{1}\chi _{E} {({\frac{x}{L}})}dx = L \int_{0}^{{{1}/{L}}}{\chi_{E}(x)}dx 
 = L |{E}\cap I|
\geq \gamma,
$$
we have 
$$K_1 \geq {1 \over {C(\gamma)}} {1 \over L} {\sum_{|n| \geq M(L)}{|f_n(y)|^2}} \geq {1 \over {C(\gamma)}} \int_{0}^{{{1}/{L}}} \left( \sum_{|n|\geq M(L)}{|f_n(x)|^2} \right)dx,
$$
where   $y \in I$ is chosen so that the last inequality holds.

Second,
\begin{align*}
|K_2| &= \bigg|\sum_{|n|, |m|\geq M(L)} \left( \int_{0}^{y} + \int_{y}^{{{1}/{L}}}
\right) \left( {\int_{y}^{x}{(f_n \bar{f_m})'(t)dt}} \right) \chi_{E}(x) e^{2 \pi i (\lambda_n - \lambda_m)x} dx \bigg| \\
& \leq \Bigg| \sum_{|n|,|m| \geq M(L)}{\int_{0}^{y}{ \left( f_n \bar{f_m} \right)' (t) \int_{0}^{t}{\chi_{E}(x) e^{2 \pi i (\lambda_n - \lambda_m)x}}dxdt}} \Bigg| \\
&+ \Bigg| \sum_{|n|,|m|\geq M(L)}{\int_{y}^{{{1}/{L}}}{ \Big( f_n \bar{f_m} \Big)' (t) \int_{t}^{{{1}/{L}}}{\chi_{E}(x) e^{2 \pi i (\lambda_n - \lambda_m)x}}dxdt}} \Bigg| =: K_{21} + K_{22}.
\end{align*}
We observe that
\begin{align*}
 K_{21} &= \bigg| \int_{0}^{y}
 \Big(%\sum_{\substack{|n|, |m|\geq M(L)\\n\neq m}
\sum_{{}^{|n|,|m|\geq M(L)}_{n\neq m}}+\sum_{{}^{|n|,|m|\geq M(L)}_{n=m}}\Big)\left( f_n ' (t) \bar{f_m} (t) +  (\bar{f_m})'(t) f_n (t)\right)\\
& \qquad\qquad\qquad\qquad\qquad\times\widehat{\chi_{E \cap [0, t]}} (\lambda_n - \lambda_m)dt \bigg|\leq K_{211}+K_{212}.
\end{align*}
We begin  with the estimate of $K_{212}$:
\begin{align*}
    K_{212} &=\bigg| \int_{0}^{y}{\sum_{|n|\geq M(L)}{\left( f_n ' (t) \bar{f_n} (t) +  (\bar{f_n})'(t) f_n (t)\right) |E\cap [0,t]|dt}} \bigg|\\
    &\leq  L^{-1} \bigg| \int_{0}^{y}{\sum_{|n|\geq M(L)}{|f_n ' (t)|^2+ |{f_n} (t)|^2}} \bigg|,
\end{align*} 
where we have used that $|E\cap [0,t]|\leq L^{-1}$.

To find the desired  bound of $K_{211}$, by the  Cauchy-Schwarz inequality,
\begin{align*}
     K_{211} &\leq 2 \int_{0}^{y}{\left( \sum_{|m|\geq M(L)}{|f_m '|^2} \right)^{1 \over 2} \left( \sum_{|n|\geq M(L)}{|f_n|^2} \right)^{1 \over 2}}\\&\times  \left( \sum_{\substack{|n|, |m|\geq M(L)\\n\neq m}}{ \big| \widehat{\chi_{E \cap [0, t]}} (\lambda_n - \lambda_m)} \big|^2 \right)^{1 \over 2}dt.
\end{align*}

We now observe that, since
$(\lambda_n)$ is Zygmund lacunary,
the set $(\lambda_n - \lambda_m)_{\substack{n,m \\n \neq m}}$ can be written as the union of $N+1$ uniformly discrete sets. Then,  by the Plancherel-Polya inequality, we know that for any $g$ with $\spec g \subset[0,1]$
$$\sum_{n \neq m} \left |\widehat{g}(\lambda_n- \lambda_m)\right|^2 \lesssim \norm{g}^2_2.$$ In particular,
\begin{align*}
    \left( \sum_{\substack{|n|, |m|\geq M(L)\\n\neq m}}{ \big| \widehat{\chi_{E \cap [0, t]}} (\lambda_n - \lambda_m)} \big|^2 \right)^{1 \over 2}
    \leq C \| \chi_{E \cap [0, t]} \|_{2} \leq 
Ct^{1 \over 2} \leq C L^{-\frac12}.
\end{align*} 

Thus,
\begin{align*}K_{211} &\leq C L^{-\frac12} \int_{0}^{{{1}/{L}}}{\left( \sum_{|m|\geq M(L)}{|f_m '(t)|^2} \right)^{1 \over 2} \left( \sum_{|n|\geq M(L)}{|f_n(t)|^2} \right)^{1 \over 2}} dt \\&\leq CL^{-\frac12} \int_{0}^{{{1}/{L}}}\left( \sum_{|m|\geq M(L)}{|f_m '(t)|^2 + |f_m(t)|^2} \right) dt,
\end{align*}
which completes the estimate of $K_{21}$.
 $K_{22}$  can be estimated similarly, using
$$\left( \sum_{\substack{|n|, |m|\geq M(L)\\n\neq m}}{ \big| \widehat{\chi_{E \cap [t, L^{-1}]}} (\lambda_n - \lambda_m)} \big|^2 \right)^{1 \over 2} \leq C L^{-\frac12}.
$$
\end{proof}
The next result confirms Conjecture 
\ref{conj:Kovrizhkin}
 for  "tail" series. 

\begin{proposition}\label{vspom}
    Let $E$ be ($\Delta,\gamma$)-thick, and let $\Lambda$ be an
 ($M,N$)-strong Zygmund lacunary sequence. Then there exists $L(\Delta, \gamma,N)$ such that Conjecture \ref{conj:Kovrizhkin} holds for $F^c=\bigcup_{k=M(L)}^\infty [\lambda_k,\lambda_{k}+1].$
\end{proposition}
\begin{proof}
First we write 
$$ \mathbb{R} = \bigsqcup_{k \in \mathbb{Z}} \big[ k \Delta, (k+1) \Delta \big] = :\bigsqcup_{k \in \mathbb{Z}} J_{k}.$$
Divide each $J_{k}$ into intervals of length $1/L$, namely, $(I_{k}^{s})_{s=1}^{S}$ with $S=L \Delta$, where $L$ will be chosen later.

We say that $j$ is a {\it good} number (written $j \in G$) if
$$ |I_{k}^{j} \cap E| > {\gamma \over 2} |I_{k}^{j}|.$$
Otherwise, we say  that $j$ is a {\it bad} number $(j \in B)$.
% if 
%$$ |I_{k}^{j} \cap E| \leq {\gamma \over 2} |I_{k}^{j}|$$

We observe that 
\begin{align}
    \# G + \# B = L \Delta
\end{align}

and 
\begin{equation}
    \label{1} \# G > C(\gamma) L\Delta  .
    \end{equation}
 Indeed,  from our assumptions,
$$ \# G \cdot {1 \over L} + \# B {{\gamma \over 2} \over L} \geq \gamma \Delta$$
and hence \eqref{1} holds for 
$C(\gamma) = {{\gamma \over 2} \over {1 - {\gamma \over 2}}}.$

We apply Lemma \ref{lemmamain} to each $I_{k}^{j}$ for $j \in G$ as follows 
\begin{align}
\label{2terms}
\int_{I_{k}^{j} \cap E} \big| \sum_{|n|\geq M(L)}f_n (x) &e^{2 \pi i \lambda_n x} \big|^2 dx 
\geq C_1 ({\gamma \over 2}) \int_{I_{k}^{j}} \sum_{|n|\geq M(L)}{|f_n(x)|^2}dx \\&- {C_2 \over L^{\frac12}}\int_{I_{k}^{j}} \left( \sum_{|n|\geq M(L)}{|f_n (x)|^2 + |f_n ' (x)|^2} \right)dx.
\end{align}
We have 
\begin{align*} \int_{E}{ \bigg| \sum_{|n|\geq M(L)}{f_n (x) e^{2 \pi i \lambda_n x}} \bigg|^2 }dx& \geq \int_{
\cup_{k}\cup_{j\in G} I_{k}^{j}\cap E
}{ \bigg| \sum_{|n|\geq M(L)}{f_n (x) e^{2 \pi i \lambda_n x}} \bigg|^2 }dx
\\& =
\sum_{k \in \mathbb{Z}} \sum_{j \in G}{\int_{I_{k}^{j} \cap E} \bigg| \sum_{|n|\geq M(L)}{f_n (x) e^{2 \pi i \lambda_n x}}} \bigg|^2 dx.
\end{align*}
Denote $$\Xi := \cup_{k} 
\cup_{j\in G}I_{j}^{k}.$$
We claim that 
$$ \int_{\Xi}{\sum_{|n|\geq M(L)}{|f_n(x)|^2 dx}} \geq %C_1 \left( {\gamma \over 2} \right) 
C_3(\bar{\gamma}, \bar{\Delta}) \sum_{|n|\geq M(L)} \int_{\mathbb{R}}{|f_n|^2}.$$ Indeed, by the Logvinenko--Sereda (Theorem 
\ref{th:logvinenko}) we need to check that for any interval $M$
\begin{align}
    \label{2}
    \big| \Xi \cap  M\big| \geq \bar{\gamma} |M|, \qquad|M|=\bar{\Delta}.
\end{align}

We claim that condition \eqref{2} holds with $\bar{\Delta} := 2\Delta$. Indeed, in this case $M$ contains $J_{k^*}$ with some  $k^* = k(M)$, that is, $J_{k^*} \subset {M}$.
Then $$\Xi \cap M \supseteq \Xi \cap J_k = \cup_{j \in G}{I_{k}^{j}}$$ and 
$$|\Xi \cap M| \geq \# G{1 \over L} \geq {C(\gamma) \over 2} \bar{\Delta}.$$

To estimate the second term in \eqref{2terms}, we use Bernstein's inequality to get
$${C \over L^{\frac12}} \int_{\mathbb{R}} \left( \sum_{|n|\geq M(L)}{|f_n^{}|^2 +|f_n'|^2} \right) dx  \leq {C' \over L^{\frac12}} \int_{\mathbb{R}} \sum_{|n|\geq M(L)}{|f_n|^2}  dx.
$$
Finally, by choosing $L$ large enough, we conclude that
\begin{align}    &\int_{E}{ \bigg| \sum_{|n|\geq M(L)}{f_n (x) e^{2 \pi i \lambda_n x}} \bigg|^2 }dx
%\int_{E}{\bigg| \sum_{|n|\geq M(L)}{f_n}(x) \bigg|^2 dx} 
\\&
\geq C(\gamma, \Delta) \int_{\mathbb{R}}{\sum_{|n|\geq M(L)}|f_n(x)|^2} dx - {C' \over L^{\frac12}} \int_{\mathbb{R}} \sum_{|n|\geq M(L)}{|f_n(x)|^2}  dx \\&\geq {C(\gamma, \Delta) \over 2}\int_{\mathbb{R}} \sum_{|n|\geq M(L)}{|f_n(x)|^2} dx.
\end{align}

\end{proof}
%\section{Main result }

The proof of Theorem \ref{theorem:mainth} will follow by combining Proposition \ref{vspom} with the well-known two-constants theorem, see \cite[p. 40]{HavinJoricke}.
    \begin{lemma}
       Let $E$ be ($\Delta,\gamma$)-thick. Then, for $g\in L^2(\mathbb{R})$ satisfying $\spec(g) \subset [0, \infty],$ one has
$$2\left(\int_{\mathbb{R}}{g^2(x)}dx\right)^{1-\sigma}\left(\int_{E}{|g(x)|^2}dx\right)^{\sigma} \geq  \int_{\mathbb{R}}{|Pg|^2(x)}dx,$$
 where $0<\sigma:=\sigma(\gamma,\Delta)\leq1$ and $\widehat{Pg}(\xi)=e^{-|\xi|} \widehat{g}(\xi)$ is the Poisson transform of $g$.
    \end{lemma}
 
\begin{proof}[Proof of Theorem \ref{theorem:mainth}]

  We will show that 
   there exists $C = C(\Delta, \gamma) \in (0, 1)$  such that 
$$ C  \int_{\mathbb{R}} \big| \sum_{n=0}^{\infty}{f_n (x) e^{2 \pi i \lambda_n x}} \big|^2 dx \leq \int_{\mathbb{R}} \big| \sum_{n=0}^{\infty}{f_n (x) e^{2 \pi i \lambda_n x}} \big|^2 \chi_{E}(x) dx
$$
for any $f_n$ satisfying $\spec(f_n) \subset [0, 1]$  and $\sum_{n=0}^\infty{f_n (x) e^{2 \pi i \lambda_n x}} \in L^2(\mathbb{R})$.

    Let $$f(x)=\left(\sum_{n=0}^{M(L)-1}+\sum_{n\geq M(L)}\right){f_n (x) e^{2 \pi i \lambda_n x}}=f_1(x)+f_2(x),$$ with $L$ from Proposition \ref{vspom} and  normalize $f$ so that $\norm{f}_2=1$.
   
     First, by the two-constants  and Plancherel's theorems, 
    $$C\norm{f_1}_2 \leq \norm{P(f)}_2 \leq 2 \norm{f \chi_E}_2^{\sigma}.$$

    Second, from Proposition \ref{vspom} and the triangle inequality, 
    $$C\norm{f_2}_2 \leq \norm{f_2 \chi_E}_2 \leq \norm{f \chi_E}_2 + \norm{f_1}_2 \leq C'(\norm{f \chi_E}_2 + \norm{f \chi_E}^\sigma_2)  .$$
    Therefore,
    $$1=\norm{f}_2 \leq \norm{f_1}_2+ \norm{f_2}_2\leq C'(\norm{f \chi_E}_2 + \norm{f \chi_E}^\sigma_2),$$ that is, 
    $$\norm{f \chi_E}_2 \geq C,$$
    completing the proof. 
    \end{proof}
\section{Further remarks}
\subsection{Sets of uniqueness for lacunary spectra}
\label{uniq sets}
Recall that 
$\Sigma$ is  a uniqueness set  for $\mathcal{C}$
if $f=0$ on $\Sigma$ implies that $f\equiv 0$ for any $f\in \mathcal{C}$.
Since any  $f$ such that $\spec{f}\subset [0,1]$ is analytic, any open  set (in fact, any infinite set with accumulation points) on $\R$ is a uniqueness set. 
Here we study 
uniqueness  sets  for functions with  lacunary spectra.

%An immediate consequence from  the Logvinenko–Sereda theorem is that  any  thick set $\Sigma$ is  a uniqueness set  for $f$ such that $\spec{f}\subset [0,1]$, that is,  $f=0$ on $\Sigma$ implies that $f\equiv 0$. In fact, we can say much more on 
%sets of uniqueness for functions with  lacunary spectra. 
% In this section we study the problem on the sets of uniqueness, if  
\label{section:uniqness}
\begin{theorem}
\label{th:uniqueness}Let $f\in L^2(\R)$ be such that $\widehat{f}$ is supported on the set $$\Sigma:=
\bigcup_{n\in \N} [-\lambda_n-1,-\lambda_n]\cup[\lambda_n,\lambda_{n}+1].$$ Assume further that 
\begin{equation}\label{eq:condition onlambda}
\lambda_n(1+\frac{1}{\log\lambda_{n}})<\lambda_{n+1}(1-
\frac{1}{\log\lambda_{n+1}}
)\quad \mbox{and}\quad \sum_{n\in\N} \frac{1}{\log^{2}\lambda_{n}}<\infty.
\end{equation}
Then if $f=0$ on an open set, then $f\equiv 0$.
\end{theorem}
\begin{proof}
Assume that $f$ vanishes on the ball $|x|<r$.

Let $\phi$ be a smooth function supported on $[-1,1]$ and satisfying $\phi=1$ for $|x|<\frac12$. Set
$$\omega(x)=\sum_{n\in \N} \frac{\lambda_n}{\log \lambda_n}\phi\Big( \frac{\log \lambda_n}{\lambda_n}(x-\lambda_n)\Big).$$
   Observe that, by \eqref{eq:condition onlambda}, the supports of the summands in this series are pairwise disjoint. Hence $\|\omega'\|_\infty \leq \|\phi'\|_\infty$, so $\omega$ is a Lipschitz function. Moreover, 
    $$\int_1^\infty \frac{\omega(x)}{x^2} dx \approx \sum_{n\in\N} \frac{1}{\log^{2}\lambda_{n}}<\infty$$ 
and
\begin{equation}
    \label{property}
\frac{|x|}{\log (e+|x|)}\leq     C+\omega(|x|), \qquad x \in \Sigma. 
\end{equation}
    Then, by the Beurling-Malliavin theorem (see, e.g., \cite{seventh}), there exists a function $g$ supported on $[-r/2,r/2]$ such that $$|\widehat{g}(x)|\leq e^{-\omega (|x|)}.$$
Since  $f$ vanishes on 
$|x|<r$
and the support of $g$ is $[-r/2,r/2]$, the function $f * g$ vanishes for $|x|<\frac{r}{2}$. 

    We now proceed to show that $$F=f*g$$ is quasi-analytic, and thus vanishes everywhere. We follow the approach of the paper \cite{jayemitkovski}.
    
Observe that $|\widehat{F}(\xi)|=|\widehat{f}(\xi) \widehat{g}(\xi)|\lesssim |\widehat{f}(\xi)| e^{-\frac{|\xi|}{\log(e+ |\xi|)}},$ because $\widehat{f}(\xi)\neq 0$ only for $\xi \in \Sigma$ and, in light of \eqref{property}, $$e^{-\omega (|\xi|)} \lesssim e^{-\frac{|\xi|}{\log(e+ |\xi|)}}=:1/W(\xi), \qquad \xi \in \Sigma.$$ 
Note that $W$ 
is  non-decreasing, continuous, log-convex and $W(0)=1$, $W(\infty)=\infty$   (that is, 
$W$
satisfies conditions (1) and (2) of Theorem 1.3 in \cite{jayemitkovski}). Thus, following that article, and
defining $$M_n=\sup_{\xi\geq 1}\frac{\xi^n}{W(\xi)}$$ 
and $$\mu_n=\frac{M_{n-1}}{M_n},
$$ we have that $M_n$ is increasing and log-convex, and that $\mu_n$ is non-increasing.

Next, by Plancherel's formula,
$\norm{ F^{(n)}}_2 \lesssim (2 \pi)^n M_n$ for any $n\geq 0.$ Therefore, by using the Sobolev inequality

$$\norm{h}_\infty\lesssim \norm{h}_2 + \norm{h'}_2$$ we derive that $\norm{ F^{(n)}}_\infty \lesssim (2 \pi)^n M_{n+1}$.

Now we  in a position  to use  the Carleman-Denjoy condition for quasianaliticity of $F$, which is 
$$\sum_{n=1}^\infty \mu_n=\infty.$$ Under the conditions of Theorem 1.3 in \cite{jayemitkovski}, by Proposition 2.2 in \cite{jayemitkovski} we know that

$$\sum_{n=1}^\infty \mu_n=\infty \iff \int_1^\infty \frac{\log W(t)}{t^2} dt =\infty.$$ In our case the last condition  is valid since 
$$\int_1^\infty \frac{\log W(t)}{t^2} dt=\int_1^\infty \frac{d\xi}{\xi \log(e+\xi)}=\infty.$$

In conclusion, $F$ is a quasi-analytic function which vanishes on an open set, thus it is zero. This means that $\widehat{F}=\widehat{f} \cdot \widehat{g}$ is zero almost-everywhere. Since $\widehat{g}$ is band-limited, and thus can only have finitely many zeros on each interval, this implies that $\widehat{f}$ vanishes almost everywhere.
\end{proof}

Note that $\lambda_n= 2^{\phi(n)}$ satisfies \eqref{eq:condition onlambda} for 
$\phi(n)={n}^\alpha \log^\beta n$, $\alpha>\frac12, \beta\in \R$ or 
$\alpha=\frac12, \beta>\frac12$.

\iffalse
$\phi(n)=\sqrt{n} \log^\delta n$, $\delta>\frac12$.

{\bf !!}
Note that any sequence $\lambda$ that increases as 
$2^{\phi(n)}$, with $\phi(n)={n}^\alpha \log^\beta n$, $\alpha>\frac12, \beta\in \R$ or 
$\alpha=\frac12, \beta>\frac12$.

$\delta>\frac12$,
or faster,
satisfies \eqref{eq:condition onlambda}.
\fi

\subsection{Some remarks about strong lacunary sequences}
\label{section:4}
In this section, we show that the class of strong Zygmund lacunary sequences lies strictly between the classes of Hadamard and Zygmund lacunary sequences.
\begin{proposition}
 \label{pr:notstrong}   Define $$\lambda=(\lambda_k)_{k=1}^\infty=(\lambda_k^j)_{1\leq k<\infty, 0\leq j\leq 1} = (4^k +j k)_{1\leq k<\infty, 0\leq j\leq 1}.$$
    Then $\lambda$ is  Zygmund lacunary but not strong Zygmund lacunary.

\end{proposition}

\begin{proof}
    
Let $4^n \leq m < 4^{n+1}$, $n\geq 0$. We claim that there are at most 5 solutions to \begin{equation}
\label{check}
    m=\lambda_{k}^j - \lambda_{k'}^{j'}.
\end{equation}
First, for the case $k>k'$. If $k<n$, then $$\lambda_{k}^j - \lambda_{k'}^{j'}<4^n\leq m.$$ If $k\geq n+2$, then 
$$\lambda_{k}^j - \lambda_{k'}^{j'} \geq 4^{n+2} - 4^{n+1}-n-1\geq 4^{n+1}>m.$$
Thus, $k\in \{n,n+1\}$, which gives four possibilities.

Second, if $k=k'$, then $i=1$ and $i'=0$, so  $k=m$.
 Hence the claim is shown.
%In total there are at most 5 possible solutions to \eqref{check}.

The sequence $\lambda$  is not strong lacunary because for any $k,l$, one has $(\lambda_{k+l}^1- \lambda_{k+l}^0) -(\lambda_k^1 - \lambda^0_k )=l$, which in particular means that for any $L$ and $M$ $$\sup_{M\leq k\neq l \geq M}\# \left\{ (k',l')\geq (M,M) :|\lambda_{k} -\lambda_l -(\lambda_{k'}-\lambda_{l'})|\leq L\right\}\geq L.$$
\end{proof}

\begin{proposition}
 \label{pr:nothadamard}   There are strong Zygmund lacunary 
 sequences which grow polynomially.  In particular, they cannot  be represented as  a finite union of Hadamard sequences.

 %such that $\liminf\frac{\lambda_{n+1}}{\lambda_{n}}=1$.

\end{proposition}

\begin{proof}
   We construct a strong Zygmund lacunary set using a modification of the well-known greedy algorithm for constructing Sidon sets.

Let $M: \N \to \N$ be a given function increasing to $\infty$. 
    Set $\lambda_1=1$. Assume that $\lambda_1<\lambda_2< \dots<\lambda_{n}$ are given and let $L$ be such that $M(L)\leq n< M(L+1)$.  We take $\lambda_{n+1}$ to be the smallest integer so that $$\lambda_{n+1} \neq \lambda_i + \lambda_j -\lambda_k + p $$ for all $i,j,k\leq n$ and $|p|\leq L$. Clearly, $\lambda_{n+1}>\lambda_n$ and 
    %there is such an integer in the range $[1,2Ln^3+1]$, thus
    $$1\leq \lambda_{n+1}\leq (2L+1)n^3+1\quad \mbox{for}\quad n<M(L+1).$$
    
    We now show that such a sequence is strong Zygmund lacunary. Let $L \in \N$ and consider the sequence $(\lambda_n/L)_{n \geq M(L)}$. Let $k,l\in \N$, $k> l$, we will see that there is no other pair $k',l' \in \N$, $k'\geq l'\geq M(L)$ so that
    $|\lambda_l- \lambda_k - (\lambda_{l'}- \lambda_{k'})| \leq L$. Set $s=\max(k,l,k',l')$.
    
    If only one of the indices is equal to $s$, the assertion is clear by construction.
    
    If not, there are two possibilities: $s=k'=l'$ or $s=k'=k>l'$.
    In the former case, by construction of $\lambda_k$, $$|\lambda_l- \lambda_k - (\lambda_{l'}- \lambda_{k'})|=|\lambda_k -\lambda_l|>L.$$
     In the latter case, $$|\lambda_l- \lambda_k - (\lambda_{l'}- \lambda_{k'})|=|\lambda_l -\lambda_{l'}|>L,$$ unless $l=l'$.
    
    As a consequence, there exist strong Zygmund lacunary sequences satisfying $$\lambda_n \leq n^3 F(n)$$ for any arbitrary increasing function $F$ with 
$\lim_{n\to \infty}
F(n)=
\infty$. Such sequences cannot be expressed as a finite union of Hadamard sequences.

\end{proof}

%\subsection{Open questions}
%QUEStion 1\ref{conj:Kovrizhkin}
%QUEStion 2 remove the condition $\Lambda\subset \mathbb{R}_+$ in Theorem \ref{theorem:mainth}.

%Question 3. Prove Theorem \ref{th:uniqueness} for $f$ vanishing on set of positive measure.
\vspace{5mm}
\bibliographystyle{abbrv}
\bibliography{references.bib}
\end{document}